\documentclass[a4paper,10pt]{amsart}
\usepackage{amsmath,amssymb,color,multirow}
\usepackage[T1]{fontenc}
\usepackage[colorlinks]{hyperref}
\definecolor{linkblue}{RGB}{1,1,190}
\definecolor{citered}{RGB}{190,1,1}
\hypersetup{linkcolor=linkblue,urlcolor=linkblue,citecolor=citered}

\theoremstyle{plain}
\newtheorem{theorem}{\bf Theorem}[section]
\newtheorem{lemma}[theorem]{\bf Lemma}

\theoremstyle{definition}
\newtheorem{remark}[theorem]{\bf Remark}
\newtheorem{conjecture}[theorem]{\bf Conjecture}

\hypersetup{hypertexnames=false}
\numberwithin{equation}{section}

\makeatletter\@namedef{subjclassname@2020}{\textup{2020} Mathematics Subject Classification}\makeatother
\allowdisplaybreaks

\begin{document}
\title[A counterexample to the AAC-Conjecture]{A counterexample to the\\ Conjecture of Ankeny, Artin and Chowla}
\author{Andreas Reinhart}
\address{Institut f\"ur Mathematik und Wissenschaftliches Rechnen, Karl-Franzens-Universit\"at Graz, NAWI Graz, Heinrichstra{\ss}e 36, 8010 Graz, Austria}
\email{andreas.reinhart@uni-graz.at}
\thanks{This work was supported by the Austrian Science Fund FWF, Project Number P36852-N}

\begin{abstract}
Let $p$ be a prime number with $p\equiv 1\mod 4$, let $\omega=\frac{1+\sqrt{p}}{2}$, let $\varepsilon>1$ be the fundamental unit of $\mathbb{Z}[\omega]$ and let $x$ and $y$ be the unique nonnegative integers with $\varepsilon=x+y\omega$. The Ankeny-Artin-Chowla-Conjecture states that $p$ is not a divisor of $y$. In this note, we provide and discuss a counterexample to this conjecture.
\end{abstract}

\maketitle

\section{Introduction}\label{1}

Let $d\geq 2$ be a squarefree integer, let $K=\mathbb{Q}(\sqrt{d})$ and let $\mathcal{O}_K$ be the ring of algebraic integers of $K$. Set

\[
\omega=\begin{cases}\sqrt{d}&\textnormal{if }d\equiv 2,3\mod 4,\\\frac{1+\sqrt{d}}{2}&\textnormal{if }d\equiv 1\mod 4,\end{cases}\quad\textnormal{and}\quad\mathsf{d}_K=\begin{cases}
4d&\textnormal{if }d\equiv 2,3\mod 4,\\ d&\textnormal{if }d\equiv 1\mod 4.\end{cases}
\]

It is well-known that $\mathcal{O}_K=\mathbb{Z}[\omega]=\mathbb{Z}\oplus\omega\mathbb{Z}$. Let $\varepsilon\in\mathcal{O}_K$ be the fundamental unit with $\varepsilon>1$. Let $x,y\in\mathbb{N}_0$ be such that

\[
\varepsilon=x+y\omega.
\]

Let $\mathbb{P}$ denote the set of prime numbers and let ${\rm h}(d)$ be the class number of $K$. This note revolves around the following conjecture.

\begin{conjecture}[The Conjecture of Ankeny, Artin and Chowla, or AAC-Conjecture]\label{Conjecture 1.1} If $d=p\in\mathbb{P}$ and $p\equiv 1\mod 4$, then $p\nmid y$.
\end{conjecture}

The AAC-Conjecture goes back to a paper of N. C. Ankeny, E. Artin and S. Chowla in 1952 (see \cite[page 480]{AnArCh52}) where the authors posed Conjecture~\ref{Conjecture 1.1} as a question. This question was later rephrased as a conjecture by A. A. Kiselev and I. Sh. Slavutski\u{\i} in 1959 \cite{KiSl59} and independently by L. J. Mordell in 1960 \cite{Mo60}. The main problem that motivated this conjecture is the computation of the class number ${\rm h}(p)$ in terms of rational numbers. For instance, if $d=p\in\mathbb{P}$ with $p\equiv 1\mod 4$, then ${\rm h}(p)y\equiv (2x+y)B_{(p-1)/2}\mod p$ by \cite{Sl04}, where $B_n$ is the $n$th Bernoulli number for each $n\in\mathbb{N}_0$. Also note that ${\rm h}(p)<p$ \cite{Sl04} and $B_n\in\mathbb{Q}$ for each $n\in\mathbb{N}_0$. If $p\nmid y$, then ${\rm h}(p)$ can be determined from the aforementioned congruence.

The Conjecture of Ankeny, Artin and Chowla has a long history and was investigated in a large number of research papers. For instance, see \cite{Ag16,BeWiZa71,BeMo24,FeMu23,Fe25,HaJo18,Ha01,Ha07,KiSl59,Lu92,Mo60,VaTeWi01,VaTeWi03,SiSh24,StWi88} (to mention a few) and the survey article by I. Sh. Slavutski\u{\i} \cite{Sl04}. The AAC-Conjecture was known to be true up to $2\cdot 10^{11}$ (see the tables below) and for a long time it was believed to hold in general. Despite substantial progress towards a proof, there were also good reasons to believe that it is false. First doubts came up by a heuristic argument (\cite[page 82]{Wa82}) which indicates that infinitely many counterexamples may exist (under the assumption that $B_{(p-1)/2}$ modulo $p$ is random for each prime number $p$ with $p\equiv 1\mod 4$). Moreover, the analogue of the AAC-Conjecture fails for fake real quadratic orders \cite{Wa17}. Another strong argument against the AAC-Conjecture is the existence of composite squarefree integers $d$ such that $d\mid y$ \cite{Re23,StWi88}. Our own doubts grew bigger after having obtained a counterexample to a related conjecture in \cite[Example 3.1]{Re24} and \cite[Theorem 2.2]{Rei24} (see the discussion in Section~\ref{4}). Thus, we started with an ``AAC-Conjecture specific'' computer search, which ended up with Theorem~\ref{Theorem 2.3}. We end this section with a brief historic overview on the verification of the AAC-Conjecture.

\begin{center}
{\it Verification history for the AAC-Conjecture}
\end{center}

\begin{table}[htbp]
\centering
\begin{tabular}{|c|c|c|}
\hline
Upper bound for $p$ & Investigator(s) & Year\\
\hline
$2000$ & Ankeny, Artin, Chowla \cite{AnArCh52} & 1952\\
\hline
$100000$ & Goldberg \cite{Mo60,Mo61} & $1954$\\
\hline
$6270714$ & Beach, Williams, Zarnke \cite{BeWiZa71} & $1971$\\
\hline
$10^9$ & Stephens, Williams \cite{StWi88} & $1988$\\
\hline
$10^{11}$ & van der Poorten, te Riele, Williams \cite{VaTeWi01} & $2001$\\
\hline
$2\cdot 10^{11}$ & van der Poorten, te Riele, Williams \cite{VaTeWi03} & $2003$\\
\hline
$1.5\cdot 10^{12}$ & Reinhart$^*$ \cite{Re23} & $2023$\\
\hline
$5.325\cdot 10^{13}$ & Reinhart$^*$ \cite{Re24} & $2024$\\
\hline
\end{tabular}
\end{table}

(*) Note that no independent double check of the search interval was done here.

\section{Main theorem}\label{2}

In this section, we provide the counterexample to the AAC-Conjecture. We start with some preliminaries.

\begin{lemma}\label{Lemma 2.1}
Let $c>1$ be an integer satisfying $2^{c-1}\equiv 1\mod c$, and assume that there exists a divisor $a$ of $c-1$ satisfying $a^2>c$ and ${\rm gcd}\left(2^{\frac{c-1}{p}}-1,c\right)=1$ for each prime number $p$ dividing $a$. Then $c\in\mathbb{P}$.
\end{lemma}

\begin{proof}
This follows from \cite[Proposition 2.1]{Rei24} (and is based on the more general results of \cite{BrLeSe75} and \cite{Po14}).
\end{proof}

We say that a subring $\mathcal{O}$ of $K$ with quotient field $K$ is an {\it order} in $K$ if it is a finitely generated abelian group with respect to addition. For each $f\in\mathbb{N}$, let $\mathcal{O}_f=\mathbb{Z}+f\mathcal{O}_K$ and observe that $\mathcal{O}_f$ is the unique order in $K$ with conductor $f$ (i.e., $\{z\in\mathcal{O}_f:z\mathcal{O}_K\subseteq\mathcal{O}_f\}=f\mathcal{O}_K$). Let $\mathcal{O}^{\times}$ denote the unit group of $\mathcal{O}$ for each order $\mathcal{O}$ in $K$.

\begin{lemma}\label{Lemma 2.2}
Let $d=p\in\mathbb{P}$ and $\eta\in\mathcal{O}_p^{\times}$ be such that $1<\eta<\varepsilon^p$. Then $p\mid y$.
\end{lemma}

\begin{proof}
Obviously, $\eta\in\mathcal{O}_K^{\times}$. Since $1<\eta<\varepsilon^p$, there is some $k\in\mathbb{N}$ such that $k<p$ and $\eta=\varepsilon^k$. By \cite[Theorem 5.9.7.4]{Ha13}, we have that $(\mathcal{O}_K^{\times}:\mathcal{O}_p^{\times})\mid p$ (since $p\mid\mathsf{d}_K$). Clearly, $\min\{r\in\mathbb{N}:\varepsilon^r\in\mathcal{O}_p^{\times}\}=(\mathcal{O}_K^{\times}:\mathcal{O}_p^{\times})$ (since $\mathcal{O}_K^{\times}/\mathcal{O}_p^{\times}$ is generated by $\varepsilon\mathcal{O}_p^{\times}$). Moreover, $\varepsilon^k\in\mathcal{O}_p^{\times}$, and hence $(\mathcal{O}_K^{\times}:\mathcal{O}_p^{\times})\mid\gcd\left(k,p\right)=1$. Therefore, $\varepsilon\in\mathcal{O}_K^{\times}=\mathcal{O}_p^{\times}$, and thus $p\mid y$.
\end{proof}

Next we present the main theorem of this note. We prove it with computer assistance (in several ways) by using Mathematica 13.2.0 and Pari/GP 2.15.2. Readers who are interested in a proof that may be checked without computer assistance, can study the proof in \cite{Rei24}. (A proof of Theorem~\ref{Theorem 2.3} below can be obtained along similar lines.)

\begin{theorem}[The counterexample to the AAC-Conjecture]\label{Theorem 2.3} Let $p=331914313984493$. Then $p$ is a counterexample to the AAC-Conjecture.
\end{theorem}

\begin{proof}
Clearly, $p\equiv 1\mod 4$. To show that $p\in\mathbb{P}$, we can either apply Lemma~\ref{Lemma 2.1} multiple times or use the ``isprime'' function of Pari/GP. It remains to prove that $p\mid y$ and we accomplish this in various ways. First we use ``standard functions'' in Mathematica and Pari/GP. In Mathematica, we apply

{\tiny
\begin{enumerate}
\item[]\hspace*{-1cm} p=331914313984493;Mod[2*NumberFieldFundamentalUnits[Sqrt[p]][[1]][[2]][[2]],p]
\end{enumerate}}

to check that $y$ is divisible by $p$. (This is the case if and only if the program returns $0$.) Note that NumberFieldFundamentalUnits[Sqrt[$p$]] determines the fundamental unit of $\mathcal{O}_K$ and represents it by using the $\mathbb{Q}$-basis $\{1,\sqrt{p}\}$ of $K$. We have to multiply NumberFieldFundamentalUnits[Sqrt[$p$]][[1]][[2]][[2]] (i.e., the rational coefficient attached to $\sqrt{p}$) by $2$, in order to obtain $y$. Also note that Mod[$a$,$p$] computes the remainder of $a$ modulo $p$. In Pari/GP we can write the lines

{\tiny
\begin{enumerate}
\item[]\hspace*{-1cm} default(parisize,100000000);
\item[]\hspace*{-1cm} p=331914313984493;imag(quadunit(p))\%p
\end{enumerate}}

to perform the same task. Observe that quadunit($p$) computes the fundamental unit of $\mathcal{O}_K$ (since $p$ is the discriminant of $K$) and it is represented by the $\mathbb{Q}$-basis $\{1,\frac{1+\sqrt{p}}{2}\}$ of $K$ (since $p\equiv 1\mod 4$). Furthermore, imag(quadunit($p$)) provides the integer coefficient attached to $\frac{1+\sqrt{p}}{2}$ and $a\%p$ is the remainder of $a$ modulo $p$. For more information on the algorithms that are used by Mathematica and Pari/GP to compute the fundamental unit, we refer to the documentation of these programs.

Another approach (to prove $p\mid y$) is to implement the continued fraction algorithm in Mathematica or Pari/GP. Thereby, we determine the continued fraction expansion of $\omega$ and compute the denominators of the convergents of $\omega$ modulo $p$. This is done until we reach the last entry of the period of the continued fraction expansion of $\omega$. In Mathematica, we can write the lines

{\tiny
\begin{enumerate}
\item[]\hspace*{-1cm} p=331914313984493;f[a\texttt{\_},b\texttt{\_},c\texttt{\_},h\texttt{\_}]:=\{k=Floor[(h+b)/a];l=k*a-b;m=(c-l*l)/a;\{m,l,k,c,h\}\}[[1]];
\item[]\hspace*{-1cm} g[u\texttt{\_}]:=\{x=1;y=0;v=Floor[Sqrt[u]];z=f[2,1,u,v];While[z[[3]]!=v,\{n=y,y=Mod[y*z[[3]]+x,u],x=n,z=f[z[[1]],z[[2]],z[[4]],z[[5]]]\}];y\}[[1]];g[p]
\end{enumerate}}

and then check that the last value is $0$. To do the same with Pari/GP, we can apply the commands

{\tiny
\begin{enumerate}
\item[]\hspace*{-1cm} p=331914313984493;f(a,b,c,h)=[k=\texttt{floor}((h+b)/a),l=k*a-b,m=(c-l*l)/a,[m,l,k,c,h]][4];
\item[]\hspace*{-1cm} g(u)=[x=1,y=0,v=\texttt{floor}(sqrt(u)),z=f(2,1,u,v),while(z[3]!=v,[n=y,y=(y*z[3]+x)\%u,x=n,z=f(z[1],z[2],z[4],z[5])]),y][6];
\item[]\hspace*{-1cm} g(p)
\end{enumerate}}

Also note that the period length of the continued fraction expansion of $\omega$ is $1486413$. Finally, we present a method to prove $p\mid y$ that does not rely on continued fraction expansions. The first step is to find $u,v\in\mathbb{N}$ such that $u<10^{p-1}$, $v<10^{p-16}$ and $u^2-p^3v^2=-4$. With Pari/GP this can be done, for instance, with the commands

{\tiny
\begin{enumerate}
\item[]\hspace*{-1cm} default(parisize,100000000);
\item[]\hspace*{-1cm} p=331914313984493;u=2*real(quadunit(p))+imag(quadunit(p));v=imag(quadunit(p))/p;
\end{enumerate}}

Next we store the values $u$ and $v$ (e.g. see \href{https://imsc.uni-graz.at/reinhart/u.txt}{https://imsc.uni-graz.at/reinhart/u.txt} and \href{https://imsc.uni-graz.at/reinhart/v.txt}{https://imsc.uni-graz.at/reinhart/v.txt}) in save files. We want to emphasize that $u$ is a positive integer with $764604$ digits and $v$ is a positive integer with $764582$ digits. Now we show that $p\mid y$ (without continued fraction expansions). (Note that it does not matter how we obtained $u$ and $v$.) Let $u,v\in\mathbb{N}$ be from the save files before. Check that $u<10^{p-1}$, $v<10^{p-16}$ and $u^2-p^3v^2=-4$. (We only require a computer to check the last equality.) Clearly, $u\equiv v\mod 2$ and $vp<10^{p-1}$. Set $\eta=\frac{u+vp\sqrt{p}}{2}$. Since $u+vp$ and $u-vp$ are even, we have that $\eta=\frac{u-vp}{2}+vp\omega\in\mathcal{O}_p$ and $\frac{u-vp\sqrt{p}}{2}=\frac{u+vp}{2}-vp\omega\in\mathcal{O}_p$. Moreover, $\eta\frac{u-vp\sqrt{p}}{2}=-1$, and thus $\eta\in\mathcal{O}_p^{\times}$. Observe that $1<\eta<10^{p-1}\left(\frac{1+\sqrt{p}}{2}\right)\leq\left(\frac{1+\sqrt{p}}{2}\right)^p\leq\varepsilon^p$. Therefore, $p\mid y$ by Lemma~\ref{Lemma 2.2}.
\end{proof}

\section{Analysis of the counterexample and how it was found}\label{3}

For the computer search that lead to the discovery of $p=331914313984493$, we implemented two algorithms that are discussed in \cite{StWi88}. They are called the {\it small step algorithm} and the {\it large step algorithm}. Let $\varepsilon^{\prime}$ be the smallest positive power of $\varepsilon$ such that $\varepsilon^{\prime}\in\mathbb{Z}[\sqrt{d}]$ and let $X,Y\in\mathbb{N}$ be such that $\varepsilon^{\prime}=X+Y\sqrt{d}$. It is well-known \cite{StWi88} that $\varepsilon^{\prime}\in\{\varepsilon,\varepsilon^3\}$. If $d\mid y$, then $d\mid Y$. If $d\mid Y$, then $d\mid 3y$. Moreover, if $d\not\equiv 5\mod 8$ or $3\nmid d$, then $d\mid y$ if and only if $d\mid Y$. The large step algorithm determines whether a squarefree integer $d\geq 2$ satisfies $d\mid Y$. It has a time complexity of $O(d^{\frac{1}{4}+\delta})$ (for each positive real number $\delta$ by \cite{StWi88}) and it is used for search purposes. The small step algorithm determines whether a squarefree integer $d\geq 2$ satisfies $d\mid y$. Since it has a time complexity of $O(d^{\frac{1}{2}+\delta})$ (for each positive real number $\delta$ by \cite{StWi88}), it is (primarily) used for verification purposes. Also note that the small step algorithm is (in general) required for a proper check of the condition ``$d\mid y$'', since the condition ``$d\mid Y$'' is weaker than the condition ``$d\mid y$''. In what follows, let $e=\lfloor\sqrt{d}\rfloor$ and for all $a\in\mathbb{Z}$ and $b\in\mathbb{N}$, let ${\rm rem}(a,b)=\min\{r\in\mathbb{N}_0:a\equiv r\mod b\}$. For all $a,b,c\in\mathbb{Z}$, let ${\rm gcd}(a,b)$, resp. ${\rm gcd}(a,b,c)$ be the greatest common divisor of $a$ and $b$, resp. of $a$, $b$ and $c$. We continue with a brief discussion of the small step algorithm.

\smallskip
{\bf The small step algorithm}: Let $(P_i)_{i=0}^{\infty}$, $(Q_i)_{i=0}^{\infty}$ and $(F_i)_{i=-1}^{\infty}$ be sequences of integers defined recursively as follows. Set

\[
P_0=\begin{cases}0&\textnormal{if }d\equiv 2,3\mod 4,\\ 1&\textnormal{if }d\equiv 1\mod 4,\end{cases}\quad Q_0=\begin{cases}
1&\textnormal{if }d\equiv 2,3\mod 4,\\ 2&\textnormal{if }d\equiv 1\mod 4,\end{cases}\quad F_{-1}=1\quad\textnormal{and}\quad F_0=0.
\]

For each $i\in\mathbb{N}_0$, let $P_{i+1}=\Big\lfloor\frac{P_i+e}{Q_i}\Big\rfloor Q_i-P_i$, $Q_{i+1}=\frac{d-P_{i+1}^2}{Q_i}$ and $F_{i+1}={\rm rem}\left(\Big\lfloor\frac{P_i+e}{Q_i}\Big\rfloor F_i+F_{i-1},d\right)$. Set $s=\min\{r\in\mathbb{N}_0:P_r=P_{r+1}\textnormal{ or }Q_r=Q_{r+1}\}$. Then

\[
y\equiv\begin{cases}F_s^2+F_{s+1}^2\mod d&\textnormal{if }Q_s=Q_{s+1}.\\ (F_{s-1}+F_{s+1})F_s\mod d&\textnormal{if }Q_s\not=Q_{s+1}.\end{cases}
\]

The small step algorithm is based on the ordinary continued fraction algorithm. It computes the continued fraction expansion (CFE) of $\omega$ until it reaches about half of the period length of the CFE of $\omega$. Thanks to well-known identities (see \cite[page 621]{StWi88}) one can now determine whether $d\mid y$. Next we briefly discuss the large step algorithm.

\smallskip
{\bf The large step algorithm}: Let $L\in\mathbb{N}$ (we pick $L\in\{\lfloor 2.5\cdot\sqrt[4]{d}\rfloor,\lfloor 2.75\cdot\sqrt[4]{d}\rfloor\}$ depending on the implementation) and let $(P_i)_{i=0}^{\infty}$, $(Q_i)_{i=0}^{\infty}$ and $(F_i)_{i=-1}^{\infty}$ be sequences of integers defined recursively as follows. Set $P_0=0$, $Q_0=1$, $F_{-1}=1$ and $F_0=0$. For each $i\in\mathbb{N}_0$, let $P_{i+1}=\Big\lfloor\frac{P_i+e}{Q_i}\Big\rfloor Q_i-P_i$, $Q_{i+1}=\frac{d-P_{i+1}^2}{Q_i}$ and $F_{i+1}={\rm rem}\left(\Big\lfloor\frac{P_i+e}{Q_i}\Big\rfloor F_i+F_{i-1},d\right)$. Set $s=\min\{r\in\mathbb{N}:Q_r=1\textnormal{ or }(r\geq L\textnormal{ and }Q_r\leq e)\}$.

\smallskip
CASE 1: $Q_s=1$. Then $Y\equiv F_s\mod d$.

\smallskip
CASE 2: $s\geq L$ and $1\not=Q_s\leq e$. Let $\left(\overline{P}_j\right)_{j=0}^{\infty}$, $\left(\overline{Q}_j\right)_{j=0}^{\infty}$, $(V_j)_{j=0}^{\infty}$, $\left(\overline{E}_j\right)_{j=0}^{\infty}$ and $\left(\overline{F}_j\right)_{j=0}^{\infty}$ be sequences of integers defined recursively as follows. Let $\overline{P}_0={\rm rem}\left(P_s,Q_s\right)$, $\overline{Q}_0=Q_s$, $V_0={\rm gcd}\left(Q_s,d\right)$, $\overline{E}_0={\rm rem}\left(\frac{P_sF_s+Q_sF_{s-1}}{V_0},d\right)$ and $\overline{F}_0=F_s$. For each $j\in\mathbb{N}_0$, let $G_j={\rm gcd}\left(P_s+\overline{P}_j,Q_s,\overline{Q}_j\right)$, let $a_j,b_j,c_j\in\mathbb{Z}$ be such that ${\rm gcd}\left(Q_s,\overline{Q}_j\right)\equiv a_jQ_s\mod\overline{Q}_j$ and $G_j=b_j(P_s+\overline{P}_j)+c_j{\rm gcd}\left(Q_s,\overline{Q}_j\right)$. Let $\left(A_i^{(j)}\right)_{i=0}^{\infty}$, $\left(B_i^{(j)}\right)_{i=0}^{\infty}$ and $\left(C_i^{(j)}\right)_{i=-1}^{\infty}$ be sequences of integers defined recursively as follows. Set $B_0^{(j)}=\frac{Q_s\overline{Q}_j}{G_j^2}$, $A_0^{(j)}={\rm rem}\left(P_s+\frac{Q_s}{G_j}{\rm rem}\left(a_jc_j(\overline{P}_j-P_s)+b_j\frac{d-P_s^2}{Q_s},\frac{\overline{Q}_j}{G_j}\right),B_0^{(j)}\right)$, $C_{-1}^{(j)}=1$ and $C_0^{(j)}=0$. For each $i\in\mathbb{N}_0$, let $A_{i+1}^{(j)}=\Big\lfloor\frac{A_i^{(j)}+\sqrt{d}}{B_i^{(j)}}\Big\rfloor B_i^{(j)}-A_i^{(j)}$, $B_{i+1}^{(j)}=\frac{d-\left(A_{i+1}^{(j)}\right)^2}{B_i^{(j)}}$ and $C_{i+1}^{(j)}={\rm rem}\left(\Big\lfloor\frac{A_i^{(j)}+\sqrt{d}}{B_i^{(j)}}\Big\rfloor C_i^{(j)}+C_{i-1}^{(j)},d\right)$. Let $t_j=\min\{r\in\mathbb{N}_0:0<B_r^{(j)}\leq e\}$, $\overline{P}_{j+1}={\rm rem}\left(A_{t_j}^{(j)},B_{t_j}^{(j)}\right)$, $\overline{Q}_{j+1}=B_{t_j}^{(j)}$, $V_{j+1}={\rm gcd}\left(B_{t_j}^{(j)},d\right)$,

\[
\overline{E}_{j+1}={\rm rem}\left(\frac{\left(A_{t_j}^{(j)}C_{t_j}^{(j)}+B_{t_j}^{(j)}C_{t_j-1}^{(j)}\right)\left(V_0V_j\overline{E}_0\overline{E}_j+d\overline{F}_0\overline{F}_j\right)+dC_{t_j}^{(j)}\left(V_0\overline{E}_0\overline{F}_j+V_j\overline{E}_j\overline{F}_0\right)}{V_0V_jV_{j+1}/{\rm gcd}\left(V_0,V_j\right)},d\right)\textnormal{ and}
\]

\[
\overline{F}_{j+1}={\rm rem}\left(\frac{C_{t_j}^{(j)}\left(V_0V_j\overline{E}_0\overline{E}_j+d\overline{F}_0\overline{F}_j\right)+\left(A_{t_j}^{(j)}C_{t_j}^{(j)}+B_{t_j}^{(j)}C_{t_j-1}^{(j)}\right)\left(V_0\overline{E}_0\overline{F}_j+V_j\overline{E}_j\overline{F}_0\right)}{V_0V_j/{\rm gcd}\left(V_0,V_j\right)},d\right).
\]

Let $m=\min\{r\in\mathbb{N}:$ there is some $n\in\mathbb{N}_0$ such that $n\leq s$, $\overline{Q}_r=Q_n$ and $\overline{P}_r\equiv P_n\mod Q_n\}$. Let $n\in\mathbb{N}_0$ be minimal such that $n\leq s$, $\overline{Q}_m=Q_n$ and $\overline{P}_m\equiv P_n\mod Q_n$. Then

\[
d\mid Y\quad\textnormal{ if and only if }\quad d\mid F_n\overline{E}_m-\frac{P_nF_n+Q_nF_{n-1}}{V_m}\overline{F}_m.
\]

The large step algorithm is also based on the continued fraction algorithm. It consists of two stages. In the first stage, we compute a certain number of steps (approximately given by the number $L$ above) of the CFE of $\sqrt{d}$. Each step of the CFE is a ``small step''. If the period length of the CFE of $\sqrt{d}$ is smaller than $L$, then $Y\equiv F_s\mod d$. If this is not the case, then the second stage comes into play. This stage is based on the composition of reduced binary quadratic forms (cf. Shank's infrastructure \cite{Sh72}). It enables us to perform ``large steps'' (i.e., approximately $s$ small steps combined) in the CFE of $\sqrt{d}$.

For more information on these two algorithms, we refer to \cite{StWi88}. On the one hand, we used plain versions (i.e., versions without optimizations) of both the small step algorithm and the large step algorithm. On the other hand, we used more advanced versions of the large step algorithm (with many optimizations) including a vectorized version (that involves AVX-512 instructions). For the scalar implementations, we used $L=\lfloor 2.5\cdot\sqrt[4]{d}\rfloor$ and for the vectorized implementation, we used $L=\lfloor 2.75\cdot\sqrt[4]{d}\rfloor$. All programs for the search were written in C and they were compiled with either GCC-12.2.0 or GCC-12.3.0 (with the compiler flag -O3). We only used privately owned hardware to perform the computations below. Next we summarize all types of computations that we performed in the search for squarefree integers $d\geq 2$ with $d\mid y$ so far.

\begin{itemize}
\item Search for all squarefree $2\leq d\leq 10^{10}$ with the plain version of the small step algorithm in $50$ hours with $16$ CPU cores.
\item Search for all squarefree $2\leq d\leq 1.5\cdot 10^{12}$ with the plain version of the large step algorithm in $650$ hours with $140$ CPU cores.
\item Search for all squarefree $1.5\cdot 10^{12}\leq d\leq 10^{14}$ with the advanced version/AVX-512 version of the large step algorithm in $7500$ hours with $162$ CPU cores.
\item Search for all prime $10^{14}\leq p\leq 3.4\cdot 10^{14}$ with $p\equiv 1\mod 4$ with the advanced version/AVX-512 version of the large step algorithm adapted for primes in $1200$ hours with $162$ CPU cores.
\item Search for all squarefree $10^{14}\leq d\leq 1.815\cdot 10^{14}$ with the advanced version/AVX-512 version of the large step algorithm in $7000$ hours with $170$ CPU cores.
\item Search for all squarefree $1.815\cdot 10^{14}\leq d\leq 2.3\cdot 10^{14}$ with the advanced version/AVX-512 version of the large step algorithm in $2850$ hours with $218$ CPU cores.
\end{itemize}

Next we discuss the counterexample in Theorem~\ref{Theorem 2.3}. First we provide an analysis along the lines of \cite[page 93]{Re24}. Let ${\rm N}:K\rightarrow\mathbb{Q}$ defined by ${\rm N}(a+b\sqrt{d})=a^2-db^2$ for all $a,b\in\mathbb{Q}$ be the norm map on $K$. In particular, ${\rm N}(\varepsilon)$ is the norm of the fundamental unit of $\mathcal{O}_K$. For each $f\in\mathbb{N}$, let ${\rm h}_f(d)$ be the class number of $\mathcal{O}_f$ (i.e., the number of elements of the class group of $\mathcal{O}_f)$. The integer $d$ is said {\it to satisfy $($RC\,$)$} if $\{f\in\mathbb{N}:{\rm h}(d)={\rm h}_f(d)\}=\{1\}$ (i.e., $\mathcal{O}_K$ is the only order in $K$ with relative class number $1$). It was proved in \cite[Proposition 2.2]{Re24} (based on results of \cite{ChSa14}) that $d$ satisfies (RC) if and only if ${\rm N}(\varepsilon)=1$, $d\not\equiv 1\mod 8$, $y$ is even and $d\mid y$. Let $\alpha\in\{0,1\}$ be such that $y\equiv\alpha\mod 2$ and let $\beta\in\{1,2,3,5,6,7\}$ be such that $d\equiv\beta\mod 8$. Set $s=|\{p\in\mathbb{P}:p\mid d\}|$. In what follows, we use the tables of \cite{Re24} without further mention. The entries in the table below were computed with Mathematica and Pari/GP. To compute the class number ${\rm h}(d)$ (for the specific $d$ below), we used the commands qfbclassno($d$) and quadclassunit($d$)[1] in Pari/GP and we used the command NumberFieldClassNumber[Sqrt[$d$]] in Mathematica.

\begin{table}[htbp]
\centering
\begin{tabular}{|c|c|c|c|c|c|c|c|c|}
\hline
$d$ & $d\mid Y$ & $d\mid y$ & \textnormal{(RC)} & $\alpha$ & $\beta$ & $s$ & ${\rm N}(\varepsilon)$ & ${\rm h}(d)$\\
\hline
$331914313984493$ & \textnormal{true} & \textnormal{true} & \textnormal{false} & $1$ & $5$ & $1$ & $-1$ & $3$\\
\hline
\end{tabular}
\end{table}

Let $\Omega$ be the set consisting of the $22$ known values of squarefree integers $d\geq 2$ for which $d\mid y$. Set $\Omega_{\beta}=\{x\in\Omega:x\equiv\beta\mod 8\}$. For each $\ell\in\mathbb{N}$, let ${\rm H}_{\ell}=\{x\in\Omega:{\rm h}(x)=\ell\}$.

\clearpage

\begin{table}[htbp]
\centering
\begin{tabular}{|*{16}{c|}}
\cline{1-7}\cline{9-16}
$\beta$ & $1$ & $2$ & $3$ & $5$ & $6$ & $7$ & \quad \quad & $\ell$ & $1$ & $2$ & $3$ & $4$ & $8$ & $16$ & 32\\
\cline{1-7}\cline{9-16}
$|\Omega_{\beta}|$ & $5$ & $3$ & $2$ & $1$ & $7$ & $4$ & \quad \quad & $|{\rm H}_{\ell}|$ & $6$ & $6$ & $1$ & $3$ & $4$ & $1$ & $1$\\
\cline{1-7}\cline{9-16}
\end{tabular}
\end{table}

Note that the counterexample to the AAC-Conjecture above is remarkable in several aspects. It is the first known example of a squarefree integer $d\geq 2$ such that $d\mid y$ and $d\equiv 5\mod 8$. Moreover, it is one of two known examples of a squarefree integer $d\geq 2$ such that $d\mid y$ and ${\rm N}(\varepsilon)=-1$ (where $d=5374184665$ is the other known example). Besides that, it is the only $d\in\Omega$ such that the odd part of the class number ${\rm h}(d)$ is larger than $1$. This is noteworthy, since it is conjectured (according to the Cohen-Lenstra heuristic \cite{CoLe84}) that the probability that the odd part (of the class number of a real quadratic number field) is larger than $1$ is $1-\left(\prod_{i=2}^{\infty}\frac{\zeta(i)(2^i-1)}{2^i}\right)^{-1}>0.2455$, where $\zeta$ is the Riemann zeta function. Last but not least, the counterexample is probably the smallest counterexample to the AAC-Conjecture. We want to emphasize, however, that the search interval $\{p\in\mathbb{P}:2\cdot 10^{11}\leq p\leq 3.4\cdot 10^{14},p\equiv 1\mod 4\}$ has not been (independently) double checked. (For this reason, we do not want to claim that it is the smallest counterexample.)

\section{Related problems and conjectures}\label{4}

We end this note with a few remarks and open problems. First, we recall a closely related conjecture that was recently settled. For other variants and analogues of the AAC-Conjecture we refer to \cite{HiSh22,Wa24,YuYu98}.

\smallskip
{\bf Mordell's Pellian Equation Conjecture.} If $d=p\in\mathbb{P}$ and $p\equiv 3\mod 4$, then $p\nmid y$.

This conjecture was first formulated (together with the AAC-Conjecture as the extended AAC-Conjecture) by A. A. Kiselev and I. Sh. Slavutski\u{\i} in 1959 (see \cite{KiSl59}). It was later stated independently by L. J. Mordell in 1961 (see \cite[page 283]{Mo61}). Mordell's Pellian Equation Conjecture was disproved in \cite{Re24,Rei24}.

We also want to mention the (combined) analogue of the AAC-conjecture and Mordell's Pellian equation conjecture for pure cubic number fields.

\smallskip
{\bf The Ankeny-Artin-Chowla-Mordell-Conjecture for pure cubic number fields.} Let $p\in\mathbb{P}$ be such that $p\geq 5$, let $d\in\{p,2p\}$, let $L=\mathbb{Q}(\sqrt[3]{d})$, let $\eta>1$ be the fundamental unit of $\mathcal{O}_L$ (i.e., the ring of algebraic integers of $L$) and let $a,b,c\in\mathbb{N}$ be such that $\eta=\frac{1}{3}(a+b\sqrt[3]{d}+c\sqrt[3]{d^2})$. Then $p\nmid b$.

This conjecture was introduced in \cite{HiSh22} and is still open. An analogue for squarefree integers fails, though. We provide the list of all squarefree integers $d$ with $2\leq d\leq 10^7$ and $d\mid b$ ($b$ is defined as below). The first $8$ values below are well-known (see Western Number Theory Problems, 17 to 19 Dec 2019, 019:03).

\begin{remark}\label{Remark 3.1} Let $d\geq 2$ be a squarefree integer, let $L=\mathbb{Q}(\sqrt[3]{d})$, let $\mathcal{O}_L$ be the ring of algebraic integers of $L$, let $\eta>1$ be the fundamental unit of $\mathcal{O}_L$ and let $a,b,c\in\mathbb{N}$ be the unique positive integers with $\eta=\frac{1}{3}(a+b\sqrt[3]{d}+c\sqrt[3]{d^2})$. Then $\{d:2\leq d\leq 10^7$ is squarefree and $d\mid b\}=\{d:2\leq d\leq 10^7$ is squarefree and $3d\mid b\}=\{3,6,15,39,42,57,330,1185,28131,47019,89411,125265,144147,435498,1688610,4580214,5123415\}$.
\end{remark}

The values in Remark~\ref{Remark 3.1} were determined with Pari/GP. Despite the fact that the AAC-Conjecture and Mordell's Pellian Equation Conjecture have been refuted, there are still interesting open problems involving squarefree $d\geq 2$ with $d\mid y$. We say that $f\in\mathbb{N}$ is {\it powerful} if $f=a^2b^3$ for some $a,b\in\mathbb{N}$.

\smallskip
{\bf The Conjecture of Erd\"os, Mollin and Walsh or EMW-Conjecture.} For each $a\in\mathbb{N}$, there is some $b\in\{a,a+1,a+2\}$ such that $b$ is not powerful.

The EMW-Conjecture was mentioned in \cite{Er75} and has subsequently been rediscovered in \cite{MoWa86}. It was shown in \cite{MoWa86} that there is some $a\in\mathbb{N}$ such that $a$, $a+1$ and $a+2$ are powerful if and only if there are some $d,k,u,v\in\mathbb{N}$ such that $d$ is squarefree, $d\equiv 7\mod 8$, $k$ and $v$ are odd, $u$ is powerful, $d\mid v$ and $\varepsilon^k=u+v\sqrt{d}$. Now we discuss the problem that motivated us to search for squarefree $d\geq 2$ with $d\mid y$.

\smallskip
{\bf Sets of distances.} Let $L$ be an algebraic number field, let $\mathcal{O}$ be an order in $L$ and let $\mathcal{A}$ be the set of nonzero nonunits of $\mathcal{O}$. A $u\in\mathcal{A}$ is called an {\it atom} of $\mathcal{O}$ if $u$ is not a product of two nonunits of $\mathcal{O}$. Each $a\in\mathcal{A}$ is a finite product of atoms of $\mathcal{O}$. For each $a\in\mathcal{A}$, let $\mathsf{L}(a)=\{k\in\mathbb{N}:a$ is a product of $k$ atoms$\}$, called the {\it set of lengths of $a$} and let $\Delta(a)=\{s-r:r,s\in\mathsf{L}(a),r<s,\{n\in\mathsf{L}(a):r\leq n\leq s\}=\{r,s\}\}$, called the {\it set of distances of $a$}. We set $\Delta(\mathcal{O})=\bigcup_{b\in\mathcal{A}}\Delta(b)$, called the {\it set of distances of $\mathcal{O}$}. We say that $\mathcal{O}$ is {\it half-factorial} if $\Delta(\mathcal{O})=\emptyset$ and in this case we set $\min\Delta(\mathcal{O})=0$. The half-factorial orders have been characterized recently in \cite{Ra23}. If $\mathcal{O}$ is seminormal (i.e., for all $x\in L$ with $x^2,x^3\in\mathcal{O}$, $x\in\mathcal{O}$) or $L$ is an imaginary quadratic number field, then $\min\Delta(\mathcal{O})\leq 1$ (see \cite{Re23}).

Now let $L=K$ be a real quadratic number field. Then $\min\Delta(\mathcal{O})\leq 2$ (see \cite{Re23}) and $\mathcal{O}$ is called {\it unusual} if $\min\Delta(\mathcal{O})=2$. Unusual orders have been completely characterized in \cite[Theorems 2.9 and 4.4]{Re23} and it was shown in \cite[Example 3.2]{Re23} that unusual orders do exist. It is possible to connect unusual orders to the condition $d\mid y$ (see \cite[Proposition 2.4 and Theorem 2.5]{Re24}). But apart from $d=5374184665$, we do not know if any of these situations can actually occur. This leads us to the following open problem.

\smallskip
{\bf Open problem.} Find squarefree integers $d\geq 2$ that satisfy one of the following conditions.
\begin{enumerate}
\item[(a)] There are $p,q\in\mathbb{P}$ such that $p\equiv 5\mod 8$, $q\equiv 3\mod 4$, $d=pq$, ${\rm h}(d)=2$, $y$ is odd and $d\mid y$.
\item[(b)] There are $p,q\in\mathbb{P}$ with $p\not=q$, $p\equiv q\equiv 1\mod 4$, $d=pq\equiv 5\mod 8$, ${\rm h}(d)=2$, ${\rm N}(\varepsilon)=-1$, $d\mid y$.
\item[(c)] There are $p,q\in\mathbb{P}$ such that $p\equiv 1\mod 8$, $q\equiv 3\mod 4$, $d=pq$, ${\rm h}(d)=2$, $y$ is odd and $d\mid y$.
\item[(d)] There are distinct $p,q\in\mathbb{P}$ such that $p\equiv q\equiv 3\mod 8$, $d=2pq$, ${\rm h}(d)=2$ and $d\mid y$.
\end{enumerate}

\noindent {\bf ACKNOWLEDGEMENTS.} We want to thank A. Geroldinger and the referees for many useful suggestions, remarks and comments that tremendously improved the quality of this note.

\end{document}